\documentclass{amsart}
\usepackage{amssymb}
\usepackage{amsthm}
\usepackage[utf8]{inputenc}
\usepackage{epsfig}
\usepackage{graphicx}
\usepackage{epstopdf}
\usepackage[usenames]{color}
\usepackage{xcolor}

\newcommand\mut[1]{\ignorespaces}

\newtheorem{thm}{Theorem}
\newtheorem{lem}[thm]{Lemma}

\def\F{\mathbb{F}}

\def\N{\mathbb{N}}
\def\Z{\mathbb{Z}}

\renewcommand\paragraph[1]{\subsection*{#1}}

\title{Factoring numbers with elliptic curves}

\author{Jorge Jim\'enez Urroz }
\address{Departamento de Matem\'aticas \\
Universidad Polit\'ecnica de  Madrid\\
Madrid, Spain}
\email{jorge.urroz@upc.edu, }
\author{Jacek Pomyka{\l}a}
\address{Institute of Mathematics, Faculty of Mathematics, Informatics and Mechanics \\
University of Warsaw \\
Poland}
\email{pomykala@mimuw.edu.pl}

\begin{document}
\begin{abstract} In the present paper we provide a probabilistic polynomial time algorithm that reduces the complete factorization of any squarefree integer $n$ to counting points on elliptic curves modulo $n$, succeeding with probability $1-\varepsilon$, for any $\varepsilon >0$.
\end{abstract}
\maketitle

\section{introduction}
The importance of factorization is something already we all agree. Since the creation of RSA cryptosystem, and its strong consequences in society, not only mathematicians, but even the public community follow the updates in the problem with great enthusiasm. The security of our systems depend heavily in the difficulty of finding the factors of certain known integer $n$.

\

Despite the efforts of the best mathematicians in the area, the fastest  algorithms for factoring, the general number field sieve,  suceeds  in finding the factors only in subexponential time which, in particular, makes nowadays extremely hard  to factor numbers of $1024$ digits. It would need more than $1$ year of continuos working of $1,5$ million computers with $2.1$Ghz of speed. And still, given the delicacy of the matter, the NIST recommends keys of $2048$ bits at least since $2015$.

\

One way of dealing with the factorization is to find problems with equivalent difficulty, i.e. mathematical problems which solutions would give a polynomial time algorithm to factor integer numbers. The first of these instances is due to Miller \cite{miller}, who proved under GRH that factoring an integer $n$ is equivalent to compute the Euler function $\varphi(n)$.

\

Since then, there have been many attempts to find equivalent problems. Bach, et all  in \cite{bach1} relate factorization of an integer $n$ with the problem of counting the sum of the divisors of $n$,  in  \cite{bach2} consider the computation of cyclotomic polynomials, and in  \cite{bach3} provides a probabilistic polynomial time algorithm to reduce factorization to  the discrete logarithm problem on composite integers $n$. It is worth to note that in all these problems, isolating certain multiple of $(p-1)$ is of interest, for some $p|n$. 

\

In $1987$ the strategy changes with the introduction of elliptic curves into the problem with the extraordinary paper by Lenstra \cite{lenstra}, in which he produces an algorithm to factor the integer $n$ by using the number of points of an elliptic curve modulo $n$. Since then, many articles relate elliptic curves with factorization. In \cite{japos} the authors prove that counting points on elliptic curves modulo $n$ is randomly computationally equivalent to factoring $n$, assuming somehow uniform distribution of the order of points of random elliptic curves over the Hasse interval. In \cite{se-pa-jo} the authors relate the problem of factoring with computing the order of points on elliptic curves modulo $n$ and give a random polynomial time reduction between both problems, while  \cite{oka-uchi} provides an algorithm reducing in polynomial time factorization to  the computation of the exponent of the group of points $E(\Z/n\Z)$ for elliptic curves modulo $n$.

\

In \cite{lu-jo} the authors prove a deterministic algorithm to reduce the factorization of any RSA modulus, $n$, to counting affine points on elliptic curves modulo $n$. For general squarefree the authors of  \cite{dr-po}, among other things, give an algorithm reducing  the factorization of $n$ to counting points on elliptic curves modulo $n$, valid for $n$ outside an exceptional set depending on the number of prime factors of $n$. For the particular  case of RSA modulus this set is of size $o(x \log\log x/\log x)$. The argument is based on the idea of finding traces of elliptic curves coprime to $p+1$ for some prime factor $p|n$.  However, to exploit the idea,  they also need to assume uniform distribution of the traces of elliptic curves, in a particular way as $p$ varies. 

\

The goal of this paper is to give a probabilistic polynomial time  algorithm succesfull with probability $1-\varepsilon$ for any $\varepsilon>0$, valid for all squarefree integers $n$ with no exception. Concretely we prove the following result.
\begin{thm}\label{thm:main} Let $n$ be a squarefree integer. Then,  assuming GRH, counting the number of points on elliptic curves modulo $n$ allows to find the complete factorization of $n$ with probability bigger than $1-\varepsilon$ for any $\varepsilon>0$.
\end{thm}
The argument also uses the idea of counting traces coprime to $p+1$ where $p|n$, but instead we do not need to assume extra hypothesis on the traces. Instead, our argument uses the unconditional bounds of Lenstra for the distribution of the number of points  as $E$ varies  randomly over all  elliptic curves over $\F_p$. To do so, Lenstra uses the class number formula, and the bounds for the special value $L(1,\chi)$ where $\chi$ is the quadratic character associated to the Frobenius field when varying the curve. In this sense one is tempted to assume that no Siegel zero exists, but this only would  save a logarithm with no effect in the result. On the other hand, we do need to assume  the general Riemann Hypothesis, GRH,  to bound the size of the least non square residue modulo $p$. 

 
\section{Lemmata}
We want to factorize $n=p_1\dots p_r$ an squarefree integer. As usual,  for $E/\Z$ we let  $E_n=\prod_{p|n}E_p$, where $E_p$ is the reduction of $E$ modulo $p$. Clearly we can give to $E_n$ group structure isomorphic to $\prod_{p|n}E_p$ as groups.  We denote $|E_n|$ the number of points of the elliptic curve modulo $n$. Now,   for any  $(d,n)=1$ consider $E_n^d$ the quadratic twist of $E_n$ and let $|E_n^d|$ its number of points modulo $n$. Then,
$$
|E_n|=\prod_{p|n}(p+1-a_p),\quad |E_n^d|=\prod_{p|n}\left(p+1-\left(\frac{d}{p}\right)a_p\right),
$$
where $\left(\frac{d}{p}\right)$ is the Legendre symbol, and $a_p$ is the trace of the Frobenius endomorphism $\varphi_p:E_p\to E_p$ which verifies $|a_p|\le 2\sqrt p$ by Hasse's Theorem \cite{hasse}. We will denote $I_p=[p+1-2\sqrt p,p+1+2\sqrt p]$ the Hasse's interval.  For the proof of the theorem we need the following lemmata.

\begin{lem}\label{lem:quad} Let $m$ be an integer and $p$ a  prime $p\nmid m$. Assuming GRH,  there exist $d\ll (\log mp)^2$ such that $\left(\frac dp\right)=-1$ and $\left(\frac dm\right)=1$.
\end{lem}
\begin{proof} This is just a consequence of Theorem 1.4 of \cite{sound-lim-li}, considering the coset $aH$ of the subgroup $H\subset (\Z/pm)^*$ of the quadratic residues modulo $pm$, and  $a$  a quadratic non-residue modulo $p$ and a quadratic residue modulo $m$.
\end{proof} 
\begin{lem} Let $D\in\N$, $p$ a prime number  and  $S_{p,D}$ be the number of isomorphic classes over $\F_p$  of elliptic curves over $\F_p$ with trace $a$ such that the gcd $(a,p+1)\le D$. Then
$$
S_{p,D}\ge C \max\left\{\frac{p}{\log p}-\frac{p}{D\log p}\tau(p+1)-\tau(p+1)^2\frac{\sqrt p}{\log p}, \frac{p}{(\log p)^2}\right\},
$$
for some explicit constant $C$.
\end{lem}
\begin{proof} By Proposition $1.9$ of \cite{lenstra}, we know that given any set $S\subset I_p$ with $|S|\ge 3$, then the number of  isomorphism classes over $\F_p$ of elliptic curves modulo $p$ with number of points in $S$ is bigger than $|S|\frac{c\sqrt p}{\log p}$, for some explicit constant $c$. So we just have to count the number of integers $1\le a\le 2\sqrt p $ such that  $(a,p+1)=d$ for some $d\le D$. Let us call $\varphi_{p,D}$ to this number. Note that the number of traces with $(a_p,p+1)\le D$ is exactly 2$\varphi_{p,D}$ since  $(a,p+1)=d$ if and only if $(-a,p+1)=d$.  Now, 
\begin{eqnarray*}
\varphi_{p,D}&=&\sum_{\substack{d|p+1\\ d\le D}}\sum_{\substack{a\le 2\sqrt p \\ (a,p+1)=d}}1=\sum_{\substack{d|p+1\\ d\le D}}\sum_{\substack{a\le 2\sqrt p/d}}\sum_{k|(a,(p+1)/d)}\mu(k)\\
&=&\sum_{\substack{d|p+1\\ d\le D}}\sum_{\substack{k|(p+1)/d}}\mu(k)\sum_{\substack{a\le \frac{2\sqrt p}{kd}}}1
=\sum_{\substack{d|p+1\\ d\le D}}\sum_{\substack{k|(p+1)/d}}\mu(k)\left[\frac{2\sqrt p}{kd}\right]\\
&=&2\sqrt p \sum_{\substack{d|p+1\\ d\le D}} \frac1{d} \sum_{\substack{k|(p+1)/d}}\frac{\mu(k)}{k}-\sum_{\substack{d|p+1\\ d\le D}}\sum_{\substack{k|(p+1)/d}}\mu(k)\left\{\frac{2\sqrt p}{kd}\right\}\\
&=&\frac{2\sqrt p}{p+1}\sum_{\substack{d|p+1\\ d\le  D}}\varphi((p+1)/d)-
\sum_{\substack{d|p+1\\ d\le D}}\sum_{\substack{k|(p+1)/d}}\mu(k)\left\{\frac{2\sqrt p}{kd}\right\},
\end{eqnarray*}
and noting that 
$$
\sum_{\substack{d|p+1}}\varphi((p+1)/d)=\sum_{\substack{d|p+1}}\varphi(d)=p+1.
$$
we get 
\begin{equation}\label{eq:fi}
\varphi_{p,D}=2\sqrt p-\frac{2\sqrt p}{p+1}\sum_{\substack{d|p+1\\ d>  D}}\varphi((p+1)/d)-
\sum_{\substack{d|p+1\\ d\le D}}\sum_{\substack{k|(p+1)/d}}\mu(k)\left\{\frac{2\sqrt p}{kd}\right\}.
\end{equation}
Using the trivial upper bound $\varphi(n)<n$ we get for the first sum in (\ref{eq:fi}) 
$$
\frac{2\sqrt p}{p+1}\sum_{\substack{d|p+1\\ d>  D}}\varphi((p+1)/d)<{2\sqrt p}\sum_{\substack{d|p+1\\ d>  D}}\frac1d< \frac{2\sqrt p}D\tau(p+1).
$$

The last sum in (\ref{eq:fi})  can be bounded in the following way
 \begin{eqnarray*}
&&\left|\sum_{\substack{d|p+1\\ d\le D}}\sum_{\substack{k|(p+1)/d}}\mu(k)\left\{\frac{2\sqrt p}{kd}\right\} \right|\le \sum_{\substack{d|p+1\\d\le D}} 2^{\omega((p+1)/d)}=\sum_{\substack{d|p+1\\d\ge (p+1)/D}} 2^{\omega(d)}\\
&\le &\prod_{q^\alpha||p+1}\sum_{\substack{d|q^\alpha}} 2^{\omega(d)}=\prod_{q^\alpha||p+1} (1+2\alpha)=\tau((p+1)^2).
\end{eqnarray*}
and hence
\begin{equation}\label{eq:upb1}
\varphi_{p,D}\ge 2\sqrt p-\frac{2\sqrt p}D\tau(p+1)-\tau((p+1)^2).
\end{equation}
On the other hand, we trivially have 
$$
\varphi_{p,D}\ge \varphi_{p,1}
$$
and 
\begin{eqnarray*}
\varphi_{p,1}&=&\sum_{\substack{a\le 2\sqrt p \\ (a,p+1)=1}}1=\sum_{\substack{a\le 2\sqrt p \\ (a,p+1)=1\\2\nmid a}}1=\sum_{\substack{a\le 2\sqrt p\\ 2\nmid a}}\sum_{k|(a,p+1)}\mu(k)=\sum_{\substack{k|p+1\\2\nmid k}}\mu(k)\sum_{\substack{a\le \frac{2\sqrt p}{k}\\ 2\nmid a}}1\\
&=&\sum_{\substack{k|p+1\\2\nmid k}}\mu(k)\left[\frac{\sqrt p}{k}+\frac12\right]=\sqrt p\sum_{\substack{k|p+1\\2\nmid k}}\frac{\mu(k)}{k}+\frac12\sum_{\substack{k|p+1\\2\nmid k}}\mu(k)+\sum_{\substack{k|p+1\\2\nmid k}}\mu(k)\left\{\frac{2\sqrt p}{d}+\frac12\right\}\\
&=&\sqrt p\frac{\varphi{(P)}}{P}+\sum_{k|P}\mu(k)\left\{\frac{2\sqrt p}{d}+\frac12\right\}
\end{eqnarray*}
where $p+1=2^kP$ with $2\nmid P$. As before, for the last sum we have trivially the upper bound 
$$
\left|\sum_{k|P}\mu(k)\left\{\frac{2\sqrt p}{d}+\frac12\right\}\right|\le 2^{\omega(P)}
$$
since the sum only counts squarefree integers, and hence
\begin{equation}\label{eq:upb2}
\varphi_{p,D}\ge \sqrt p \frac{\varphi(P)}{P}-2^{\omega(P)}.
\end{equation}

Now, if $\omega(P)=l$, then $P\ge \prod_{p\le p_l}p$. From Corollary $2$ of Rosser and Shoenfeld \cite{rosho} we obtain that $p_l>el$ for any $p_l>113$, so in particular $\prod_{p\le p_l}p\ge l^l$. Indeed we just need Maple to check it is true for $13 \le l\le 31$ and for $l>31$ we will prove it by induction. So, suppose it is true that $\prod_{p\le p_{l-1}}p\ge (l-1)^{(l-1)}$. Then, 
$$
\prod_{p\le p_{l}}p=p_l\prod_{p\le p_{l-1}}p\ge p_l(l-1)^{(l-1)}>el(l-1)^{(l-1)},
$$
but since $a_l=(1+\frac1l)^l$ is an increasing sequence with $\lim_{n\to\infty}a_l=e$, we have 
$$
el(l-1)^{(l-1)}>l(1+\frac1{l-1})^{l-1}(l-1)^{(l-1)}=l^l
$$
Hence, taking $l=16$ we see than $P\ge l^l>16^l=(2^l)^4$ so, in particular,  we have $P>(2^{\omega(P)})^4$ and
$$
\varphi_{p,1}> \sqrt p \frac{\varphi(P)}{P}-p^{1/4}\ge \frac{\sqrt p}{8\log p},
$$
for any $p\ge 10^8$, where we have used the trivial inequality $\varphi(n)>\frac{n}{4\log n}$.
\end{proof}

\section{Proof of Theorem \ref{thm:main}}

Let $D\le (\log n)^k$ for some fixed $k$.  We select an elliptic curve $E_n:=y^2=x^3+ax+b$  modulo $n$ at random. Suppose $(a_{p_1},p_1+1)=d_E\le D$. Take $d$ as in Lemma \ref{lem:quad}, for $p=p_1$ and $m=n/p_1$. Then, $|E_{p_1}^d|=p_1+1+a_{p_1}$ while 
$|E_{p_i}^d|=p_i+1-a_{p_i}$ for any $i=2\dots r$ and, hence
$$
\frac{|E_n|}{|E_n^d|}=\frac{p_1+1-a_{p_1}}{p_1+1+a_{p_1}}.
$$
Let $\frac{|E_n|}{|E_n^d|}=\frac ab$ where $(a,b)=1$. Then
$$
\frac ab=\frac{\frac{p_1+1-a_{p_1}}{d}}{\frac{p_1+1+a_{p_1}}{d}},
$$
and hence $a=(p_1+1-a_{p_1})d$, $b=(p_1+1+a_{p_1})d$. So we just have to mulptiply $a$ by all the integers up to $D$ in order to recover $p_1+1-a_{p_1}$  and mulptiply $b$ by all the integers up to $D$ to recover $p_1+1+a_{p_1}$. Adding both number we get $2(p_1+1)$ and, hence, we have found a factor of $n$.

\

If the algorithm does not return $p_1$, in $(\log n)^k$ steps, then stop, and select a new elliptic curve $E'_n:=y^2=x^3+Ax+B$ non isomorphic to $E_n$. Note that  we can assume $A\not\equiv 0\pmod p$ neither $B\not\equiv 0\pmod p$ for $p|n$ since otherwise we find a factor of $n$.  Then, if we would get an elliptic curve isomorphic to $E_n$ over $\F_p$ for some $p|n$ then $\lambda^4A\equiv a\pmod p$ and $\lambda^6B\equiv b\pmod p$, so $\lambda^4\equiv \frac{a}{A}$ and $\lambda^6\equiv \frac{b}{B}\pmod p$, or $\lambda^2\equiv \frac{bA}{aB} \pmod p$ and then $b^2A^3-a^3B^2\equiv 0\pmod p$ and we would find a factor.

\

Now, repeating the process $k(\log n)^{2}$ times, for  non isomorphic classes of elliptic curves, the probality that none of them has the trace $a$ so that  $(a,p+1)\le D$ is bounded above by
$$
\left(1-\frac{c}{(\log p)^2}\right)^{k(\log n)^{2}}<\ \left(1-\frac{c}{(\log p)^2}\right)^{k(\log p)^{2}}<e^{-\frac{k}{c}}<\frac{c}{k}<\varepsilon,
$$
for $k\gg (1/\varepsilon)$. Finally note that checking that none of them is isomorphic to the previous ones would take $O\left(k^2(\log n)^{4}\right)$ steps.

\section{Statements and Declarations}

The first author is supported by  the  Grant PID2019-110224RB-I00 of the Spanish Government. The authors have no relevant financial or non-financial interests to disclose. All authors contributed to the study, conception and design of the paper. All authors read and approved the final manuscript.
\end{document}